\documentclass[11pt,
]{amsart}
\usepackage{
amssymb, tikz-cd
}
\newcommand{\no}[1]{#1}
\renewcommand{\no}[1]{}  \newcommand{\upDelta}{\Delta} 
\no{\usepackage{times}\usepackage[
subscriptcorrection, slantedGreek, 
nofontinfo]{mtpro}
\renewcommand{\Delta}{\upDelta}
}
\date{\today}

\numberwithin{equation}{section}% 

\newtheorem{theorem}{Theorem}[section]
\newtheorem{proposition}{Proposition}[section]

\newtheorem{definition}{Definition}[section]
\newtheorem{corollary}{Corollary}[section]

\theoremstyle{definition}
\newtheorem{remark}{Remark}[section]

\DeclareMathOperator{\Vol}{Vol}

\DeclareMathOperator{\WF}{WF}

\newcommand{\eps}{\varepsilon}

\newcommand{\R}{{\bf R}}
\newcommand{\Id}{\mbox{Id}}
\renewcommand{\r}[1]{(\ref{#1})}
\newcommand{\PDO}{$\Psi$DO}
\newcommand{\be}[1]{\begin{equation}\label{#1}}
\newcommand{\ee}{\end{equation}}

\renewcommand{\d}{\mathrm{d}}

\newcommand{\bo}{\partial M}

\newcommand{\B}{B}

\title[The geodesic ray transform  on Riemannian surfaces with conjugate points]{The geodesic ray transform on Riemannian surfaces  with conjugate points}

\author[F. Monard]{Fran\c{c}ois Monard}
\address{Department of Mathematics, University of Washington, Seattle, WA 98195}
\thanks{First author partly supported NSF grant No.~1265958}

\author[P. Stefanov]{Plamen Stefanov}
\address{Department of Mathematics, Purdue University, West Lafayette, IN 47907}
\thanks{Second author partly supported by a NSF  Grant DMS-1301646}

\author[G. Uhlmann]{Gunther Uhlmann}
\address{Department of Mathematics, University of Washington, Seattle, WA 98195}
\thanks{Third author partly supported NSF grant No.~1265958 and a Simons Fellowship}

\usepackage{graphicx}
\DeclareGraphicsExtensions{.pdf,.png,.jpg}
\begin{document}

\begin{abstract}
We study the geodesic X-ray transform $X$ on compact Riemannian surfaces with conjugate points. Regardless of the type of the conjugate points, we show that we cannot recover the singularities and therefore, this transform is always unstable (ill-posed). We describe the microlocal kernel of $X$ and relate it to the conjugate locus. We present numerical examples illustrating the cancellation of singularities. We also show that the attenuated X-ray transform is well posed if the attenuation is positive and there are no more than two conjugate points along each geodesic; but still ill-posed, if there are three or more conjugate points. Those results follow from our analysis of the weighted X-ray transform. 
%Several examples are also studied.
\end{abstract}

\maketitle

\section{Introduction} The purpose of this paper is to study the X-ray transform on Riemannian surfaces over geodesics  with  conjugate points. 
 Let $\gamma_0$ be a fixed directed geodesic  on a Riemannian manifold of dimension $n\ge2$, and let $f$ be a function which  support does not contain the endpoints of $\gamma_0$. We study first the following \textit{local  problem}: what part of  the wave front set $\WF(f)$ of $f$ can be obtained from knowing the wave front of the (possibly  weighted) integrals 
\be{01}
Xf(\gamma) =\int \kappa(\gamma(s),\dot\gamma(s)) f(\gamma(s))\, \d s
\ee
of $f$ along all (directed) geodesics $\gamma$ close enough to $\gamma_0$? The analysis can be easily generalized to more general geodesic-like curves as in \cite{FSU} or to the even more general case of ``regular exponential maps''  \cite{Warner_conjugate}  as in \cite{SU-caustics}. For the simplicity of the exposition, we consider the geodesic case only. 
Since $X$ has a Schwartz kernel with singularities of conormal type, $Xf$ could only provide information about $\WF(f)$ near the conormal bundle ${N}^*\gamma_0$ of $\gamma_0$. It is well known that if there are no conjugate points on $\gamma_0$, we can in fact recover $\WF(f)$ near $N^*\gamma_0$. This goes back to Guillemin \cite{Guillemin85, GuilleminS} for integral transforms satisfying the Bolker condition, see also \cite{Quinto94}. The latter allows the use of the clean intersection calculus of Duistermaat and Guillemin \cite{DuistermaatG} to show that $X^*X$ is a pseudo-differential operator (\PDO), elliptic when $\kappa\not=0$. 
      In the geodesic case under consideration, a microlocal study of $Xf$ has been done in \cite{SU-Duke, SU-JAMS, FSU, SU-AJM, SU-caustics,UV:local}, and in some of those works,  $f$ can  even be a tensor field. 

In the presence of conjugate points, the Bolker condition fails, see also Proposition~\ref{pr-Bolker}. Then $X^*X$ is not a \PDO\ anymore, and the standard arguments do not apply. Moreover, one can easily construct examples with $\kappa=1$  based on the sphere \cite{SU-caustics} where the localized X-ray transform has an infinitely dimensional kernel of  distributions singular at or near $N^*\gamma_0$. Two delta functions of opposite signs placed at two anti-podal points is one such case. In \cite{SU-caustics}, the second and the third author studied this question assuming that  the conjugate points on $\gamma_0$ are of fold type, i.e., $v\mapsto \exp_xv$ has fold type singularities \cite{Arnold_caustics, Golubitsky_G}. In dimension $n=2$, we proved there that there is a loss of $1/4$ derivatives, and we describes the microlocal kernel of $X$ modulo that loss.  The analysis there was based on the specific properties of the fold conjugate points, and was focused on understanding the structure of $X^*X$ as a sum of a \PDO\  plus a  Fourier Integral Operator (FIO) associated with the conormal bundle $N^*\Sigma$ of the conjugate locus $\Sigma$, see next section. The microlocal structure of $X^*X$ in all dimensions was also studied by Sean Holman \cite{Sean-X}.  
 In this work, we show that there is a loss of all derivatives, and that this is true for all possible types of conjugate points. Instead of studying $X^*X$, we study $X$ itself as an FIO; show that the singularities in the kernel are related by a certain FIO;  and describe its canonical relation as a generalization of $N^*\Sigma'$ even if the latter may not be smooth.  

Since the answer to the local problem in two dimensions is affirmative when there are no conjugate points, and we show that it is negative when there are, the present paper gives a complete answer to the local problem for $n=2$ ($Xf$ known near a single $\gamma_0$) with the exception of the borderline case when the conjugate points are on $\bo$.

The \textit{global problem}, recovery of $\WF(f)$, and ultimately $f$ from $Xf$ known for all (or for a ``large'' set of) geodesics is different however. Let $M$ be two dimensional and non-trapping so that $Xf$ is defined globally. The union of $N^*\gamma$ for all directed geodesics in $M$ is a double cover of $T^*M\setminus 0$, and for each $(x,\xi) \in T^*M\setminus 0$, $\WF(f)$ there can be possibly resolved by $Xf$ known near each of the two directed geodesics through $x$ normal to $\xi$. We therefore have a system of two equations, and if the weight is not even with respect to  the direction, then that system is  solvable for a pair of conjugate points, provided that some determinant does not vanish, see \r{R1}. If the weight is even, then the global problem is equivalent to the local one. 
Condition \r{R1} is naturally satisfied for the geodesic attenuated X-ray transform with positive attenuation. 
Those two equations however are not enough to resolve the singularities at three points conjugate to each other, and then we still cannot recover $\WF(f)$.  
Those results are formulated in Section~\ref{sec5a}. 

Recovery of $\WF(f)$ is directly related to the question of stability of the geodesic  X-ray transform (or any other linear map) --- given $Xf$, can we recover $f$ in a stable way? 
We do not study the uniqueness question here but we will just mention that if $(M,g)$ is real analytic, in many cases, even ones with conjugate points, one can have injectivity based on analytic microlocal arguments. An example of that is a small tubular neighborhood of $\gamma_0$ as above;  then the geodesics normal or close to normal to $\gamma_0$  carry enough information to prove injectivity based on analyticity arguments as in \cite{FSU, SU-AJM, BQ1, Venky09}, for example. Stability however, will be always lost for even weights, for example, if there are conjugate points. 
The attenuated X-ray transform is stable with one or no conjugate point along each geodesic, and unstable otherwise. The term stable still makes sense even  if there is no injectivity; then it indicates that  estimate \r{5.1cc} holds.

This linear problem is also in the heart of the non-linear problem of recovery a metric or a sound speed (a conformal factor) from the lengths of the geodesics measured at the boundary (boundary rigidity) 
or from knowledge of the lens relation (lens rigidity), see, e.g., \cite{Croke04,Croke91,CDS, PestovU,Sh-book,SU-lens, SU-Kawai,SU-JAMS,SU-MRL}; or from knowledge of the hyperbolic Dirichlet-to-Neumann (DN) map \cite{BelishevK92, BellassouedDSF,Carlos_12,Sh-2D,SU-IMRN,SU-JFA,BaoZhang}. It is the linearization of the first two ($f$ is a tensor field then); and the lens relation is directly related to the DN map and its canonical relation as an FIO. 
Although fully non-linear methods for uniqueness (up to isometry) exist, see, e.g., \cite{BelishevK92}, stability is always derived from stability of the linearization. Very often, see for example \cite{SU-JAMS}, even uniqueness is derived from injectivity \textit{and stability} of the linearization, see also \cite{SU-JFA} for an abstract treatment.  Understanding the stability of $X$ is therefore fundamental for all those problems. 

In seismology, recovery of the jumps between layers, which mathematically are conormal singularities, is actually the main goal. The model then is a linearized map like $X$ or a linearized DN map; and the goal is to recover the visible part of the wave front set. 

In dimensions $n\ge3$, the problem is over-determined. If there is a \textit{complete} set of geodesics without conjugate points which conormal bundle covers $T^*M$, then $\WF(f)$ can be recovered \cite{SU-AJM}. In case this is not true or when we have local information only, there can be instability, for example for metrics of product type. In \cite{SU-caustics}, we formulated a  condition for fold conjugate points, under which singularities can still be recovered because then the ``artifact'' (the $F_{21}$ term in Theorem~\ref{thm_cancel}) 
are of lower order. At present it is not clear however if there are metrics satisfying that condition; and even of they are, the analysis does not cover non-fold conjugate points. Therefore, this problem remains largely open in dimensions $n\ge3$. 

The structure of the paper is as follows. In Section~\ref{sec_RM}, we present some facts about the structure of the conjugate points. In Section~\ref{sec3}, we characterize $X$ as an FIO. The main theoretical results are in Section~\ref{sec4} and Section~\ref{sec5a}. Numerical evidence is presented in Section~\ref{sec5}.

\section{Regular exponential maps and their generic singularities}  \label{sec_RM}
\subsection{Regular exponential maps} 
Let $(M,g)$ be a fixed $n$-dimensional Riemannian manifold. We recall some known facts about the structure of the conjugate points, mostly due Warner \cite{Warner_conjugate}. Most of the  terminology comes from there but some is taken from \cite{SU-caustics}.

\subsection{Generic properties of the conjugate locus}
We recall here the main result by Warner  \cite{Warner_conjugate} about the regular points of the conjugate locus of a fixed point $p$. In fact, Warner considers more general exponential type of maps but we restrict our attention to the geodesic case (see also \cite{SU-caustics} for the general case). 
The \textbf{\emph{tangent conjugate locus}} $S(p)$ of $p$  is the set of all vectors $v\in T_pM$ so that $\d_v\exp_p( v)$ (the differential of $\exp_p(v)$ w.r.t.\ $v$) is not an isomorphism.
We call such vectors \textbf{\emph{conjugate vectors}} at $p$ (called conjugate points in \cite{Warner_conjugate}).  The kernel of $\d_v\exp_p(v)$ is denoted by $N_p(v)$. It is part of $T_vT_pM$ that we identify with $T_pM$. By the Gauss lemma, $N_p(v)$ is orthogonal to $v$. 
The images of the conjugate vectors under the exponential map $\exp_p$ will be called \textbf{conjugate points} to $p$. The image of $S(p)$ under the exponential map $\exp_p$ will be denoted by $\Sigma(p)$ and called the \textbf{conjugate locus of $p$}, whoc may not be smooth everywhere. Note that $S(p)\subset T_pM$, while $\Sigma(p)\subset M$. We always work with $p$ near a fixed $p_0$ and with $v$ near a fixed $v_0$. Set $q_0=\exp_{p_0}(v_0)$. Then we are interested in $S(p)$ restricted to a small neighborhood of $v_0$, and in $\Sigma(p)$ near $q_0$. Note that $\Sigma(p)$ may not contain all points near $q_0$ conjugate to $p$ along some geodesic; and may not contain even all of those along $\exp_{p_0}(tv_0)$ if the later self-intersects --- it contains only those that are of the form $\exp_p(v)$ with $v$ close enough to $v_0$.

We denote by $\Sigma$ the set of all conjugate pairs $(p,q)$ localized as above. In other words, $\Sigma=\{(p,q);\; q\in \Sigma(p)\}$, where $p$ runs over a small neighborhood of $p_0$. Also, we denote by $S$ the set $(p,v)$, where $v\in S(p)$. 
 
A \textbf{\emph{regular conjugate vector}}  $v$ at $p$ is defined by the requirement that  there exists a neighborhood $U$  of $v$, so that any radial ray $\{tv\}$ of $T_pM$ contains at most one conjugate vector  in $U$.  The regular conjugate locus then is an everywhere dense open subset of the conjugate locus and  is an embedded  $(n-1)$-dimensional manifold. The order of a conjugate vector as a singularity of $\exp_p$ (the dimension of the kernel of the differential) is called an order of the conjugate vector. 

In \cite[Thm~3.3]{Warner_conjugate}, Warner characterized the  conjugate vectors at a fixed $p_0$ of order at least $2$, and some of those of order $1$, as described below. Note that in $(\B_1)$, one needs to postulate that $N_{p_0}(v)$ remains tangent to $S(p_0)$ at points $v$ close to $v_0$ as the latter is not guaranteed by just assuming that it holds at $v_0$ only. 

\smallskip  
%\begin{description}
\textbf{($\mathbf{F}$)  Fold conjugate vectors.} Let $v_0$ be a regular  conjugate vector at $p_0$, and let $N_{p_0}(v_0)$ be one-dimensional and transversal to $S(p_0)$.
Such singularities are known as fold singularities.   
Then one can find local coordinates $\xi$ near $v_0$ and $y$ near $q_0$ so that in those coordinates, $y= \exp_{p_0}\xi$ is given by 
\be{2.1}
%\xi \longmapsto (x',(x^n)^2).
y' = \xi', \quad y^n=(\xi^n)^2.
\ee
Then $\det \d_\xi \exp_{p_0}\xi = 2\xi^n$ and 
\be{2.1n}
S(p_0)=\{\xi^n=0\}, \quad N_{p_0}(v_0)=\mbox{span}\left\{ \partial/\partial \xi^n\right\}, \quad\Sigma(p_0) = \{y^n=0\}. 
\ee
Since the fold condition is stable under small $C^\infty$ perturbations, as follows directly from the definition, those properties are preserved under a small perturbation of $p_0$. 
\smallskip  

\textbf{($\mathbf{B}_1$) Blowdown of order 1.}   
Let $v_0$ be a regular  conjugate vector at $p_0$ and let $N_{p_0}(v)$ be one-dimen\-sional. Assume also that  $N_{p_0}(v)$ is tangent to $S(p_0)$ for all regular conjugate $v$ near $v_0$. 
We call such singularities  blowdown of order 1. 
Then locally, $y= \exp_{p_0}\xi$ is represented in suitable coordinates by
\be{2.1b}
y'=\xi', \quad y^n = \xi^1\xi^n.
\ee  
Then $\det \d_\xi \exp_{p_0}\xi = \xi^1$ and 
\be{2.1bb}
S(p_0) = \{\xi^1=0\}, \quad N_{p_0}(v_0)=\mbox{span}\left\{ \partial/\partial \xi^n\right\},  \quad \Sigma(p_0) = \{y^1=y^n=0  \}. 
\ee
Even though we postulated that the tangency condition is stable under perturbations of $v_0$, it is not stable under a small perturbation of $p_0$, and  the type of the singularity may change then. In some symmetric cases, one can check directly that the type is locally preserved. 
\smallskip  

\textbf{($\mathbf{B}_k$)   Blowdown of order $k\ge2$.} Those are regular conjugate vectors in the case where $N_{p_0}(v_0)$ is  $k$-dimensio\-nal, with $2\le k\le n-1$. Then in some coordinates,  $y= \exp_{p_0}\xi$ is represented as
\be{2.1c}
\begin{split}
y^i &=\xi^i,\qquad i =1,\dots,n-k\\ 
y^i &= \xi^1\xi^i, \quad i =n-k+1,\dots,n.
\end{split}
\ee
Then $\det \d_\xi \exp_{p_0}\xi = (\xi^1)^k$ and 
\be{2.1cn}
\begin{split}
S(p_0) &= \{\xi^1=0\}, \quad N_{p_0}(v_0)=\mbox{span}\left\{\partial/\partial{\xi^{n-k+1}},\dots,\partial/\partial{\xi^n}\right\},\\
\Sigma(p_0) &= \{y^1=y^{n-k+1}= \dots =y^n=0  \}.
\end{split}
\ee
In particular, $N_{p_0}(v_0)$ must be tangent to $S(p_0)$, see also \cite[Thm~3.2]{Warner_conjugate}. This singularity is unstable under perturbations of $p_0$, as well. A typical example are  the antipodal points on $S^{n}$, $n\ge3$; then $k=n-1$. Note that in this case, the defining equation $\det \d_v \exp_{p_0}\xi=0$ of $S(p_0)$ is degenerate, and $n\ge3$. 
%\end{description}

\subsection{The 2D case} 
If $n=2$, then the order of the conjugate vectors must be one since in the radial direction the derivative of the exponential map is non-zero.  
Only (F) and ($B_1$) regular conjugate points are possible among the types listed  above. For each $p_0$ and $v_0\in S(p_0)$, then 
\be{alt}
\text{Either $N_{p_0}(v_0)$ is transversal to $S(p_0)$; or $N_{p_0}
(v_0)= T_{v_0}S(p_0)$}
\ee
In the first case, $v_0$ is of fold type. The second case is more delicate and depends of the order of contact of $N_{p_0}(v_0)$ with $S(p_0)$. To be more precise, let $\xi$ be a smooth non-vanishing vector field  along $S(p_0)$ so that at each point $v\in S(p_0)$, $\xi$ is collinear with $N_{p_0}(v)$. Let $\kappa$ be a smooth function on $T_{p_0}M$ with a non-zero differential so that $\kappa=0$ on $S(p_0)$. Then $\d \kappa(\xi)$ restricted to $S(p_0)$, has a zero at $v=v_0$, see also \cite{Golubitsky_G} for this and for the definition below. 

\begin{definition}\label{def_cusp}
If this zero is simple, we say that $v_0$ is a simple cusp. 
\end{definition}

Near a simple cusp, the exponential map has the following normal form \cite{Golubitsky_G} 
\[
y^1 = \xi^1,\quad y^2 = \xi^1\xi^2 +(\xi^2)^3.
\]
Then $\det \d_\xi \exp_{p_0}\xi = \xi^1+3(\xi^2)^2$ and
\be{2.1cusp}
\begin{split}
S(p_0) &= \{\xi^1 +3(\xi^2)^2=0\}, \quad N_{p_0}(v_0)=\mbox{span}\left\{\partial/\partial{\xi^{2}}\right\},\\
\Sigma(p_0) &= \{y^1 = -3t^2,\; y^2 = -2t^3,\; |t|\ll1 \}.
\end{split}
\ee
Note that away from $(\xi^1,\xi^2)=(0,0)$, we have a fold since then the first alternative of \r{alt} holds.  Such a cusp is clearly visible in our our numerical experiments in Figure~\ref{pic4}. 

Regardless of the type of the conjugate vector $v$, the tangent conjugate locus $S(p)$ to any $p$ is a smooth curve. This follows form  \cite{Warner_conjugate} since the order of any conjugate vector is always one. Its image $\Sigma(p)$ under the exponential map however is locally a smooth curve if $v$ is a fold, a point, if $v$ is of ($B_1$) type, and a curve with a possible singularity at $v$ when the second alternative in \r{alt} holds. %On the other hand, $\Sigma$ is always smooth. 

\section{The geodesic X-ray transform as an FIO} \label{sec3}
The properties of $X$ as an FIO have been studied in \cite{Greenleaf-Uhlmann}, including those of the restricted X-ray transform, in the framework of Guillemin \cite{Guillemin85} and Guillemin and Stenberg \cite{GuilleminS}. We will recall the results in \cite{Greenleaf-Uhlmann} with some additions. 

\subsection{The general $n$ dimensional case} Assume  first  $n:=\dim M\ge2$. 
 We extend the manifold and the metric to some neighborhood of $M$.  
We will study $X$ restricted to geodesics in some open set  $\mathcal{M}_0$ of geodesics with endpoints outside $M$.   
 It is clear (and shown below) that $\mathcal{M}_0$ is a manifold. The set of all geodesics $\mathcal{M}$ might not be, if there are trapped ones. We assume first the following
\be{odd}
\gamma(t)\in \mathcal{M}_0\quad  \Longrightarrow \quad \gamma(-t)\not\in \mathcal{M}_0.
\ee
This condition simplifies the exposition. In fact later, we study geodesic manifolds for which the opposite holds, and the general case can be considered as a union of the two. The reason for \r{odd} is to guarantee that for any $\gamma\in \mathcal{M}_0$, each element in $N^*\gamma$ is conormal to exactly one geodesic in  $\mathcal{M}_0$.
Let $M_0$ be the points   on the geodesics in $\mathcal{M}_0$, in the interior of $M$.  
We will study $X$ acting on distributions $f$ supported  $\mathcal{M}_0$. In particular, this covers the case of geodesics in some  small enough neighborhood of a fixed geodesic  $\gamma_0$ as shown in Figure~\ref{pic0}.

\begin{figure}[!ht] % float placement: (h)ere, page (t)op, page (b)ottom, other (p)age
  \centering
  \includegraphics[width = 3.5in,page=1]{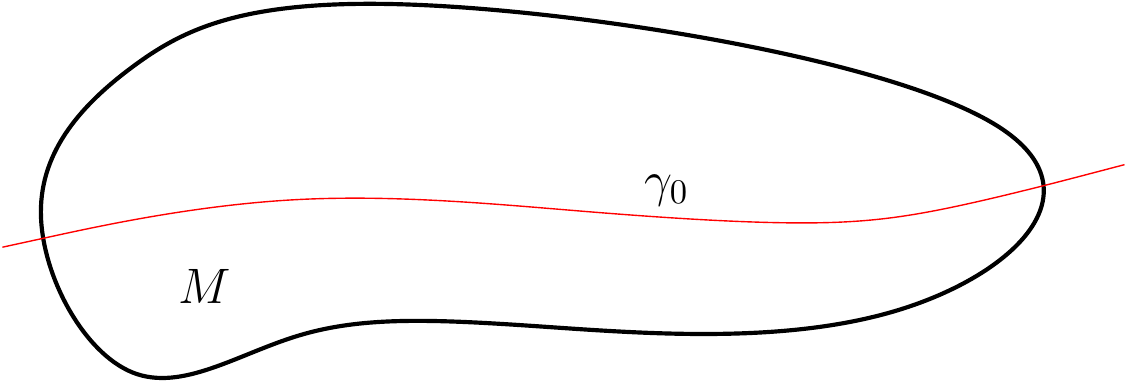}
  \caption{The manifold $M$ and a fixed geodesic $\gamma_0$.} 
%On the second row, a few families of geodesics in \texttt{dom2} are plotted and the domain is enlarged to the size of \texttt{domain1}.  }
  \label{pic0}
\end{figure}

To parameterize the (directed) geodesics near some $\gamma_0$, we choose a small oriented  hypersurface 
 $H$ intersecting $\gamma_0$ transversally.  It can be a neighborhood  of $\bo$ near $\gamma(a)$ or $\gamma(b)$ if $\gamma_0$ hits $\bo$ transversely at that particular end. 
Let $\d\Vol_H$ be the induced measure in $H$, and let $\nu$ be a smooth unit normal vector field on $H$ consistent with the orientation of $H$. Let $\mathcal{H}$ consist of all $(p,\theta)\in SM$ with the property that $p\in H$ and $\theta$ is not tangent to $H$, and positively oriented, i.e., $\langle\nu,\theta \rangle>0$.  Introduce the measure $\d\gamma = \langle n,\theta\rangle\,\d\Vol_H(p)\,\d \sigma_p(\theta) $ on $\mathcal{H}$. 
Then one can parametrize all geodesics intersecting $H$ transversally by their intersection $p$ with $H$ and the corresponding direction, i.e., by elements in $\mathcal{H}$. An important property of $\d\gamma$ is that it introduces a measure on that geodesics set that is invariant under a different choice of $H$ by the Liouville Theorem, see e.g., \cite{SU-Duke}. 

We can project $\mathcal{H}$ onto the unit ball bundle $BH$ by projecting orthogonally each $\theta$ as a above to $T_pH$. The measure $\d\mu$ in this representation becomes the induced measure on $TH$, and in particular, the factor $\langle n,\theta\rangle$ (which does not even make sense in this representation) disappears. We think of  $BH$ as a chart for the $2n-2$ dimensional  geodesic manifold $\mathcal{M}_0$ (near $\gamma_0$).  Then $B^*H$ (which we identify with $BH$) has a natural symplectic form, that of $T^*H$. We refer to \cite{Greenleaf-Uhlmann} for an invariant definition of $\mathcal{M}_0$, in fact, they did that in two ways. It also has a natural volume form, which is invariantly defined on $\mathcal{M}_0$. Tangent vectors to $\mathcal{M}_0$  can naturally be identified with the Jacobi fields modulo the two trivial ones, $\dot\gamma(t)$ and $t\dot\gamma(t)$.

We view $X$ as the following (continuous) map
\[
 X:C_0^\infty(M_0) \to C^\infty(\mathcal{M}_0). 
\]
  The Schwartz kernel of $X$ is clearly a delta type of distribution which indicates that it is also a (conormal) Lagrangian distribution, and therefore $X$ must be an FIO. To describe it in more detail, we use the double fibration point of view of \cite{GelfandGS_69,Helgason-Radon}.  
  Let 
\[
Z_0 = \{  (\gamma,x)\in \mathcal{M}_0\times M_0;\; x\in\gamma   \}
\]
be the point-geodesic relation with $\pi_{\mathcal M}: Z_0\to \mathcal{M}_0$ and $\pi_M: Z_0\to M_0$ being the natural left and right projections. It follows form the analysis below that $Z_0$ is smooth of dimension $2n-1$. 
 Then
\begin{center}
\begin{tikzcd}[]
&  Z_0 \arrow{dl}[swap]{\pi_{\mathcal M}}\arrow{dr}{\pi_M} & \\
\mathcal{M}_0 &&M_0
\end{tikzcd}
\end{center}
is a double fibration in the sense of \cite{GelfandGS_69}. Note that we switched left and right here compared to some other works in order to get a canonical relation later with complies with the notational convention in  \cite{Hormander4}. 
To get the microlocal version of that, following Guillemin \cite{Guillemin85}, let 
\[
N^*Z_0 = \left\{    (\gamma,\Gamma,x,\xi)\in T^*(\mathcal{M}_0\times M_0)\setminus 0 ;\; (\gamma,\Gamma)=0 \;\text{on } T_{(x,\gamma)}Z_0  \right\}
\]
be the conormal bundle of $Z_0$. It is a Lagrangian submanifold of $T^*(\mathcal{M}_0\times M_0)$, and the associated canonical relation is given by the ``twisted'' version of $N^*Z_0$:
\be{4.C}
C = N^*Z'_0 = \left\{   (\gamma,  \Gamma,x,\xi);\; (\gamma,  \Gamma,x,-\xi)\in N^*Z_0     \right\}.
\ee
We then have the microlocal version of the diagram above:
\be{4.2}
\begin{tikzcd}[]
&  C \arrow{dl}[swap]{\pi_{\mathcal M}}\arrow{dr}{\pi_M} & \\
 T^*\mathcal{M}_0&&T^*M_0  
\end{tikzcd}
\ee
where now $\pi_{\mathcal M}$ and $\pi_M$ denote the projections indicated above. It is easy to see that their ranges do not include the zero sections. 

The weighted X-ray transform $X$ then has a Schwartz kernel $\kappa\delta_{Z_0}$ defined by 
\[
\langle \kappa\delta_{Z_0},\phi \rangle =  \int_{\mathcal{M}_0 }\left(   \int_\R \kappa(\gamma(s),\dot\gamma(s))  \phi(\gamma(s), \gamma)\, \d s \right)\d\gamma .
\]
One can think of $\kappa$ as a function defined on $Z_0$. Then $\kappa \delta_{Z_0}$ is a distribution conormal to $Z_0$, see \cite[section~18.2]{Hormander3}  in the class $I^{-n/4}(M_0\times\mathcal{M}_0,Z_0)$. Therefore, 
\[
X : C_0^\infty(M_0)\longrightarrow C_0^\infty(\mathcal{M}_0)
\]
is an FIO of order $-n/4$ associated with the Lagrangian $N^*Z_0$, which can be extended to distributions, as well. Its canonical relation is given by $C$ above.

Let $H$ be as above and fix some local coordinates $y$ on $H$; and let $\eta$ be the dial variable. 
We denote below by  $\gamma(y,\eta,t) $ the geodesic issued from $(y,\tilde\eta)$, where $\tilde\eta$ is unit and has projection $\eta$ on $TH$ (with a fixed orientation). 
In those coordinates, the manifold $Z_0$ consist of all $(y,\eta,x)\in BH\times M$ with the property that $x=\gamma(y,\eta,t)$ for some $t$. We can think of this as a parametric representation 
\be{4.5}
(y,\eta,t) \longmapsto (y,\eta,\gamma(y,\eta,t)).
\ee
Then the map \r{4.5} has Jacobian 
\[ J:= 
\begin{pmatrix}
\Id & 0&0\\
0 &\Id&0\\
 \partial x/\partial y& \partial x/\partial\eta&   \partial x/\partial t 
  \end{pmatrix}.
\]
Here, 
$\partial x/\partial y = \{\partial x^i/\partial y^\alpha\}$ with $i=1,\dots,n$ and $\alpha=1,\dots,n-1$, $x=  \gamma(y,\eta,t)$, and similarly for $\partial x/\partial\eta$. The identity operators above are in $\R^{n-1}$. Since the bottom right   element, the tangent to the geodesic $ \gamma(y,\eta,t)$ is never zero, we get that the $(3n-2)\times (2n-1)$ matrix $J$ has maximal rank $2n-1$. In particular, this shows that $Z_0$ is a smooth manifold of dimension $2n-1$.

The conormal bundle $N^*Z_0$ is at any point of $Z_0$ is the space conormal to the range of $J$. Denote by $(\hat y,\hat\eta,\xi)$ the dual variables of $(y,\eta,x)$. Then  $(y,\eta, \hat y,\hat\eta,x,\xi)\in  C=N^*Z_0'$ if and only if $(y,\eta,x)\in C$ and 
\be{4.6}
\xi_i  \dot \gamma^i(y,\eta,t)  =0,\quad 
\xi_i \frac{\partial   \gamma^i(y,\eta,t) }{\partial y^\alpha}-  \hat y_\alpha =0,\quad 
\xi_i \frac{\partial  \gamma^i(y,\eta,t) }{\partial \eta^\alpha}-  \hat \eta_\alpha =0,\quad \alpha=0,\dots,n-1. 
\ee

The first condition says that $\xi$ is conormal to the geodesic issued from $(y,\eta)$, which is consistent with the fact that we can only hope to recover conormal singularities at the geodesics involved in the X-ray transform.  In particular, we get 
\be{4.6a}
\pi_M(C)=\cup_{\gamma\in \mathcal{M} } N^*\gamma := V.
\ee
What we see immediately is that we have the inclusion $\subset$ above. On the other hand, given $(x,\xi)\in V$, there is a geodesic in $\mathcal{M}_0$ normal to $(x,\xi)$ by the definition of $V$. Then we can compute $\hat y$ and $\hat \eta$ as above, see also the 2D case below. Notice that we  have an $(n-2)$-dimensional manifold of geodesics normal to $(x,\xi)$;  only one, undirected, if $n=2$. In the latter case, we may have two directed ones. 
 In particular this implies that the rank of $\d\pi_M$ is maximal and equal to $2n$ but since $\d\pi_M$ maps locally $\R^{3n-2}$ to $\R^{2n}$, it has an $n-2$ dimensional kernel, if $n\ge3$, and is an isomorphism when $n=2$. 

The next two equations in \r{4.6}  say that the projection of $\xi$ (identified with a vector by the metric) to any non-trivial Jacobi field, at the point $x= \gamma(y,\eta,t)$, is given. Set
\be{4.6b}
\mathcal{V}: = \pi_{\mathcal{M}}(C)
\ee
The projection $\pi_{\mathcal{M}}$ maps the $3n-2$ dimensional $C$ to the $4n-4$ dimensional $T^*\mathcal{M}_0$. 
We want to find out when this allows us to recover $t$, and therefore, $x$. Assume that we have two different values   $t=t_1$ and $t=t_2$ of $t$, at which \r{4.6} holds with the so given $(y,\eta,\hat y, \hat\eta)$ with $(x_1,\xi^1)$ and $(x_2,\xi^2)$, respectively.  Consider all non-trivial Jacobi fields vanishing at $t_1$. They form a linear space of dimension $n-1$. Their projections to $\xi^1$ at $t_1$ vanish in a trivial way. By \r{4.6} and by our assumptions, their projections to $\xi^2$ at $t_2$ must vanish as well. Since $\xi^2$ is not tangent to the geodesic $\gamma(y,\eta,t) $, we get that at $t_2$, those Jacobi fields form a subspace of dimension $n-2$. Therefore, a certain non-trivial linear combination  vanishes there, which means that there is a Jacobi field vanishing at both $t_1$ and $t_2$. Then the corresponding $x_1$ and $x_2$ are conjugate points, as a consequence of our assumption of lack of injectivity. 

This argument also proves that $\d\pi_{\mathcal M}$ is injective as well. Indeed, locally, there are no conjugate points. The problem then can be reduced to showing that $t=t_1$ is a simple root of the equations \r{4.6} with $(y,\eta,\hat y,\hat\eta)$ given, which follows from the fact that the zeros of Jacobian fields are always simple. 

On the other hand, assume that there are no conjugate points along the geodesic $\exp_y(t\eta)$. Then any Jacobi field vanishing at $t=t_1$ would be nonzero for any other $t$. Those (non-trivial) Jacobi fields span an $(n-1)$-dimensional linear space as above at any $t\not=t_1$. On the other hand, they are all perpendicular to $\xi^1$ at $t_1$. Since $\xi^1\not=0$, we get a contradiction.

Therefore, $\pi_{\mathcal{M}}$ is injective if and only if there are no conjugate points along the geodesics in $\mathcal{M}_0$. In particular, it is always locally injective. 
Moreover,   $\d\pi_{\mathcal M}$ is injective as well.

The X-ray transform $X$ is said to satisfy the \textit{Bolker condition} \cite{Golubitsky_G} if $\pi_{\mathcal M}$ is an injective immersion. Then $X^*X$ is a \PDO\ of order $-1$. It is elliptic at $(x,\xi)$ if and only if  $\kappa(x,\theta)\not=0$ for at least one $\theta$ with $\langle\theta, \xi\rangle=0$, see, e.g., \cite{FSU, SU-AJM}. Then one can recover singularities conormal to all geodesics over which we integrate by elliptic regularity. The analysis above yields the following. 
\begin{proposition}\label{pr-Bolker}
The Bolker condition is satisfied for $Z_0$  if and only if none of the geodesics in $\mathcal{M}_0$ has conjugate points. % in $M_0$. 
\end{proposition}

An indirect indication of the validity of this proposition is the fact that $X^*X$ is a \PDO\ if and only if there are no conjugate points, as mentioned above. The latter however was proved by analyzing the Schwartz kernel of $X^*X$ directly, instead of composing $X^*$ and $X$ as FIOs. In the more difficult case of the restricted X-ray transform ($\mathcal{M}_0$ is a submanifold of $\mathcal{M}$ of the same dimension as $M_0$, when $n\ge3$), the Bolker condition can be violated even if there are no conjugate points, for examples for the Euclidean metric \cite{Greenleaf-Uhlmann}.

We summarize the results so far, most of them due to \cite{Greenleaf-Uhlmann, Guillemin85}, in the following. Let $\mathcal{M}_0$, $M_0$, $C$ be as above.  
We recall that the zero subscript indicates that we work in an open subset of geodesics.

\begin{theorem}
$X$ is a Fourier Integral Operator in the class $I^{-\frac{n}4}( \mathcal{M}_0\times M_0,C')$. It satisfies the Bolker condition if and only if the geodesics in $\mathcal{M}_0$ have no conjugate points. 
In the latter case, $X^*X$ is a \PDO\ of order $-1$ in $M_0$. 
\end{theorem}

We also recall the result in \cite{SU-caustics} showing that if the conjugate points are of fold type, $X^*X$ has a canonical relation constituting of the following non-intersecting canonical relations: the diagonal (a \PDO\ part) and $N^*\Sigma'$, where $\Sigma$ is the conjugate locus defined as the pairs of all conjugate points, a smooth manifold in that case,  see \cite{Greenleaf-Uhlmann}.

\subsection{The 2-dimensional case} \label{sec_3.2} 
Assume now  $n:= \dim M=2$. In this case, the  three manifolds in the diagram \r{4.2} have the same dimension, $4$. A natural question is whether  $\pi_{\mathcal M}$ and $\pi_M$ are diffeomorphisms, local or global. The analysis of the $n$-dimensional case answers this already but we will make this more explicit below. 

We introduce  the scalar Jacobi fields $a$ and $b$ following  \cite{PestovU04} below. Introduce the notation $v^\perp$, where for a given vector $v$, we define the covector   $v^\perp$ by  $v^\perp_i:= \mathcal{R}_{ij}v^j$, where
\[
\mathcal{R}_{ij} = \sqrt{\det g}\begin{pmatrix} 0&1\\-1&0\end{pmatrix}      .
\]
Note that  $v^\perp$ has the same length as the vector $v$ and is conormal to $v$. The inverse map is then given by  $\xi_\perp^i = \sum_j \mathcal{R}_{ij}^{-1}\xi_j $. If we think of $v\mapsto v^\perp$ as a rotation by $-90$ degrees, then $\xi\mapsto \xi_\perp$ is a rotation by $90$ degrees.

With $H$ as above (now, a curve transversal to $\gamma_0$), every Jacobi vector field along any fixed geodesic $\gamma$ close to $\gamma_0$ is a linear combination $\hat y a(y,\eta,t)+\hat \eta b(y,\eta,t)$, where $(\hat y, \hat \eta)$ are the initial conditions. As before, $(y,\eta)\in BH$. 
The first of the conditions \r{4.6} say that for such a point in $C$, we have $\xi= \lambda \dot\gamma^\perp$ with some $\lambda\not=0$.   
The last two equations imply  $\lambda a = \hat y$, $\lambda b = \hat\eta$ where
\[
a(y,\eta,t) =   \frac{\partial\gamma^i}{\partial y} \dot\gamma_i^\perp , \quad
b(y,\eta,t) =   \frac{\partial\gamma^i}{\partial \eta} \dot\gamma_i^\perp .
\]
The functions $a$ and $b$ are the projections of the Jacobi fields ${\partial\gamma^i}/{\partial y}$ and ${\partial\gamma^i}/{\partial \eta}$ to $\dot\gamma^\perp$. They solve the scalar Jacobi equation 
\[
\ddot a+Ka=0, \quad \ddot b+Kb=0,
 \] 
 where $K$ is the Gauss curvature, with linearly independent initial conditions. % $(1,0)$ and $(0,1)$, respectively.  
Then
\be{4.7}
C = \left\{\left(  y,\eta,      \lambda   a(y,\eta,t), \lambda  b(y,\eta,t),    \gamma(y,\eta,t),\lambda\dot\gamma^\perp(y,\eta,t) \right)|\; (y,\eta) \in BH,\; \lambda\not=0, t\in\R  \right\},
\ee
compare with \r{4.2}. %, where $g(y)$ is the induced metric on the one-dimensional $H$; in other words, $ga$ and $gb$ have to viewed as  scalar components of covectors, instead of vectors. 
  Clearly, $\dim C=4$, therefore, all manifolds in the diagram \r{4.2} are of the same dimension, 4.

The Bolker condition which we analyzed  above says that  $\pi_{\mathcal M}$  is an injective immersion and only if there are no conjugate points along  the geodesics in $\mathcal{M}_0$. In particular, this is true near $\gamma_0$ if there are no conjugate points on $\bar\gamma_0$. We will prove this again in this 2D situation. 
Fix $(y,\eta,\hat y,\hat \eta)$ with $(\hat y,\hat\eta)\not=0$. We want to see first if the time $t$ for which
\be{4.8}
\lambda a(y,\eta,t)=\hat y, \quad \lambda b(y,\eta,t) = \hat \eta
\ee
is unique. 
Consider the scalar non-zero Jacobi field $c(y,\eta,t): =  \hat y b(y,\eta,t)-\hat\eta a(y,\eta,t) $ that vanishes when \r{4.8} holds. The problem is reduced to showing uniqueness of the solution to $c(y,\eta,t)=0$ with respect to $t$. If there are two solutions however, then they correspond to conjugate points. This proves the injectivity of the projection $\pi_{\mathcal M}$ in this case. The injectivity of differential of $\pi_{\mathcal M}$ follows from the fact that $\d c/\d t\not=0$ when $c=0$. Since $\dim T^*\mathcal{M}_0=\dim C=4$, we get that $\pi_{\mathcal M}$ is actually a local diffeomorphism. It is global, from $C$ to $\mathcal{V}$,  assuming no conjugate points along any geodesic in $\mathcal{M}_0$.

We show now that $\pi_M:C\to V$ is a  diffeomorphism, see  \r{4.6a}. 
For $(x,\xi)\in N^*\gamma$, for some $\gamma\in\mathcal{M}$, 
let $\gamma_{x,v}$ be the unique, by \r{odd}, geodesic with the unit $v$ so that $v=\pm (\xi/|\xi|)_\perp$. %There is at least one, by assumption, but there might be two, if inverting the direction keeps the geodesic in $\mathcal{M}$.  
Without loss of generality we may assume that  the sign above is positive. 
Let  $y\in H$  be the point where it hits $H$ for the first time, $t$, and let $\eta$ be the projection of the direction at $y$ to $TH$. Then $(y,\eta,t)$ depend smoothly on $(x,\xi)$ as a consequence of the assumption that $\gamma_0$ hits $H$ transversely. Thus the inverse of $\pi_M$ is given by 
\[
\pi_M^{-1}:   (x,\xi)\mapsto (y,\eta,\lambda a(y,\eta,t),\lambda b(y,\eta,t),x,\xi), \quad y=y(x,\xi), \; \eta=\eta(x,\xi), \; t=t(x,\xi) 
\]
 with the last three functions defined as above. So $\pi_M$ is a local diffeomorphism.  
If the opposite to \r{odd} holds, then $\pi_M$ is a double cover.

Combining this with the previous paragraph,  we get in particular that  $\mathcal{C}(x,\xi) = (y,\eta,\lambda a,\lambda b)$ given by $\mathcal{C} = \pi_{\mathcal{M}}\circ  \pi_M^{-1}$ 
is a local diffeomorphism. 

We summarize those results below. Recall that $\mathcal{M}_0$ is an open subset of geodesic and that $M_0\subset M$ consist of the interior point on those geodesics, and that  $X$ is restricted to $\mathcal{M}_0\times M_0$. 

\begin{theorem} \label{thm_2D}
Let $\dim M=2$. Then under the assumptions in this section,

(a)  $X$ is an FIO of order $-1/2$ associated with the canonical relation $C$ given by \r{4.7}, which is a graph of the canonical map $\mathcal{C}$ described above. 

(b) $\mathcal{C}$ is a local diffeomorphism. It is a global one,  from $V$ to $\mathcal V$ if and only if there are no conjugate points on the geodesics in $\mathcal{M}_0$. %and if $\mathcal{M}_0^{\rm even }$ is empty. 

(c)  If there are no conjugate points on the geodesics in $\mathcal{M}_0$, 
$X$ is elliptic at $(x,\xi)$ if and only if $\kappa(x,v)\not=0$ for $v$ such that $v^\perp$ is collinear with $\xi$.
\end{theorem}

Conjugate points destroy the injectivity of $\mathcal{C}$, while a violation of condition \r{odd} (assumed above) %destroys surjectivity; then 
makes $\mathcal{C}$  $1$-to-$2$. 

Notice also that regardless of existence or not of conjugate points, $X: H^s(M_0)\to H^{s+1/2}(\mathcal{M}_0)$ continuously  because one can use a partition of unity and reduce the problem to the case of no conjugate points, see also \cite{SU-caustics}.  Ellipticity of $X$ is understood in the sense of \cite[\S 25.3]{Hormander4}. 

%[PROVE ELLIPTICITY (easy)]

\section{Cancellation of singularities and instability} \label{sec4}
We are ready now to prove  a stronger version of the cancellation of singularities of $X$ when $\dim M=2$. First, we will prove that we have a cancellation of infinite order. Second, the type of the conjugate locus will play no role at all. 
More precisely, 
since $X$  maps $H^{s-1/2}$ to $H^{s}$ locally, resolution of singularities without loss of derivatives would mean $Xf\in H^s(\gamma_0,\Gamma^0)\Longrightarrow f\in H^{s-1/2}(x_0,\xi^0) $ for any $(\gamma_0,\Gamma^0, x_0,\xi^0)\in C$. We proved in \cite{SU-caustics} that this is not true if the conjugate points are of fold type, and there is an infinite dimensional space of distributions, for which $f\in H^{s-1/2}(x_0,\xi^0)$ but    $Xf\in H^{s+1/4}(\gamma_0,\Gamma^0)$, and this is true actually in open conic neighborhoods of those points. This is a loss of at least of $1/4$ derivatives, even if we can recover $\WF(f)$ in another Sobolev space. We show below that we have actually loss of all derivatives, and the type of the conjugate points does not matter. 

Assume from now on that $\mathcal{M}_0$ is a small neighborhood of some $\gamma_0$. Then the set $V$ has two natural disjoint components, corresponding to the choice of the orientation of the normals to the geodesics. In the representation \r{4.7}, this corresponds to the choice of the sign of $\lambda$. Assume the convention that $\lambda>0$ corresponds to the positive orientation. Then 
\be{5.0}
V=V_-\cup V_+, \quad  \mathcal{V} = \mathcal{V}_-\cup \mathcal{V}_+; \quad \mathcal{V}_\pm := \mathcal{C}(V_\pm). 
\ee

To understand better what $\mathcal{C}(p_1,\xi^1) = \mathcal{C}(p_2,\xi^2)$ means, observe first that the latter is equivalent to the following. The points $p_1$ and $p_2$ belong to the same geodesic $\gamma(y,\eta,t)$, i.e., $p_1=\gamma(t_1,y,\eta)$, $p_2 = \gamma(t_2,y,\eta)$. Next, 
\be{5.1}
\lambda_1 a(t_1,y,\eta)= \lambda_2 a(t_2,y,\eta), \quad \lambda_1 b(t_1,y_0,\eta)= \lambda_2 b(t_2,y_0,\eta),
\ee
and  
\[
\xi^1 = \lambda_1\dot\gamma^\perp(t_1,y,\eta), \quad \xi^2 = \lambda_2\dot\gamma^\perp(t_2,y,\eta).
\]
\begin{figure}[!ht] % float placement: (h)ere, page (t)op, page (b)ottom, other (p)age
  \centering
  \includegraphics[width = 3.5in,page=2]{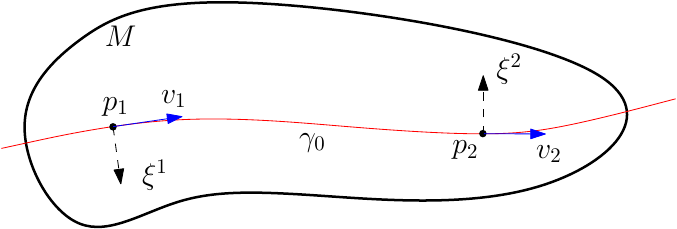}
  \caption{ Singularities that cancel; $p_1$ and $p_2$ are conjugate and there is an  even number of conjugate points  (or none) between $p_1$ and $p_2$. Next,  $(p_2,\xi^2) = \mathcal{C}_{21}(p_1,\xi^1)$, see \r{C21}. }
%On the second row, a few families of geodesics in \texttt{dom2} are plotted and the domain is enlarged to the size of \texttt{domain1}.  }
  \label{pic01}
\end{figure}
In what follows, we drop the dependence on $(y,\eta)$.  
The Wronskian $W(t):= a(t)b'(t)-a'(t)b(t)$ is independent of $t$ and therefore equal to its initial condition $W(0)\not=0$. Consider the Jacobi field $c(t) := a(t_1)b(t)- b(t_1)a(t)$ (as we did above). Then $c(t_1) =c(t_2)=0$. We have
\[
c'(t_1) = W(0), \quad c'(t_2) = (\lambda_2/\lambda_1)W(0). 
\]
Therefore, 
\[
\lambda_1 = \lambda c'(t_1), \quad \lambda_2 = \lambda c'(t_2), \quad \lambda\not=0. 
\]
%Since we can do this for any other geodesic in $\mathcal{M}_0$, 
We therefore proved the following.
\begin{theorem}\label{thm_C}
$\mathcal{C}(p_1,\xi^1) = \mathcal{C}(p_2,\xi^2)$ if and only if there is a geodesic  $[0,1]\to \gamma\in \mathcal{M}_0$ joining  $p_1$ and $p_1$ so that 

(a) $p_1$ and $p_2$ are conjugate to each other, along $\gamma$,

(b) $\xi=\lambda J'(0)$, $\eta=\lambda J'(1)$, $\lambda\not=0$, where $J(t)$ is the unique non-trivial, up to multiplication by a constant, Jacobi field with $J(0)=J(1)=0$. 
\end{theorem}

Of course, (b) is equivalent to saying that $\xi=\lambda c'(0)\dot\gamma^\perp$, $\eta=\lambda c'(1)\dot\gamma^\perp$, $\lambda\not=0$, where $c(t)$ is the unique non-trivial, up to multiplication by a constant, Jacobi field with $c(0)=c(1)=0$. In particular, $\xi$ and $\eta$ are conormal to $\gamma$. 
This generalizes a result in \cite{SU-caustics} proved there under the assumption that the conjugate points are of fold type.

From now on, we fix $(p_1,\xi^1)$ and $(p_2,\xi^2)$ so that $\mathcal{C}(p_1,\xi^1) = \mathcal{C}(p_2,\xi^2)$. 
Let $V^1=V_-^1\cup V_+^1$ be a small conic neighborhood of $(p_1,\xi^1)$, and  let $V^2=V_-^2\cup V_+^2$ be a small conic neighborhood of $(p_2,\xi^2)$.  Note that the signs of $c'(0)$ and $c'(1)$ in Theorem~\ref{thm_C} above are the opposite if the number $m$ of conjugate points between $p$ and $q$ is even (or zero);   and the signs are equal  otherwise. Let $\epsilon=(-1)^m$.  
By shrinking those neighborhoods a bit, we can assume that   $  \mathcal{C}(V^1_\pm) = \mathcal{C}(V^2_{\mp\epsilon})=  :\mathcal{V}_\pm$. Let $\mathcal{C}_k:= \mathcal{C}|_{V^k}$, $k=1,2$, where, somewhat incorrectly, $\pm(-1)^m= \pm$, if $m$ is even, and $\pm(-1)^m= \mp$. 
 Set
\be{C21}
\mathcal{C}_{21}:= \mathcal{C}_2^{-1}\circ \mathcal{C}_1 : V^1\to V^2. 
\ee
Then $\mathcal{C}_{21}$ is a canonical relation itself, and is a diffeomorphism. Also,  $\mathcal{C}_{21}: V_\pm^1\to V_{\mp  \eps}^2$.

%For the purpose of the next theorem, reall that $\Sigma\subset M\times M$ is the conjugate locus. 
Next theorem extends the corresponding result in \cite{SU-caustics} from the case of fold conjugate points (in any dimension) to any type of conjugate points (in two dimensions). It relates the canonical relation $\mathcal{C}_{21}$ directly to the geometry of the conjugate locus. 

\begin{theorem}\label{C12}\ 

(a) $\mathcal{C}_{21}: V^1\to V^2$ is a diffeomorphism.

(b) $\mathcal{C}_{21} : (p_1,\xi^1)\to (p_2,\xi^2)$ also admits the  characterization of Theorem~\ref{thm_C}. 

\end{theorem}

The following theorem describes the microlocal kernel of $X$ in this setup. 

\begin{theorem}\label{thm_cancel} 
Let $\kappa(p_2,\xi^2)\not=0$. Then  there exists an FIO $F_{21}$ of order zero with canonical relation $\mathcal{C}_{21}$ with the following property. Let  $f_k\in \mathcal{E}'(M_0)$ with $\WF(f_k)\subset V^k$, $k=1,2$, with $V_{1}$, $V^2$ small enough.  Then 
\be{5.1'}
X(f_1+f_2)\in H^s(\mathcal{V})
\ee
 if and only if 
 \be{5.1a}
 f_2+ F_{21}f_1\in H^{s-1/2}(V^2).
 \ee
  The FIO $F_{21}$ is elliptic if and only if $\kappa(p_1,\xi^1)\not=0$. 
\end{theorem}

Clearly, under the ellipticity assumptions above, we can swap the indices $1$ and $2$ to obtain
\[
X(f_1+f_2)\in H^s(\mathcal{V}) \quad \Longleftrightarrow\quad   f_2 + F_{21}f_1\in H^{s-1/2}(V^2)\quad  \Longleftrightarrow\quad  f_1+ F_{12}f_2\in H^{s-1/2}(V^1),
\]
where $F_{12}= F_{21}^{-1}$ (microlocally). 
\begin{proof}
Let $X_k$ be $X$, restricted to distributions with wave front sets in $V_k$, $k=1,2$. Then $Xf = X_1f_1+X_2f_2$. We proved above that $X_k$ are FIOs with canonical relations $\mathcal{C}_k$, $k=1,2$; elliptic, if  $\kappa(p_k,\xi^k)\not=0$. Then an application of the parametrix $X_2^{-1}$ to $Xf$ completes the proof. In particular, we get 
\[
F_{12} = X_2^{-1}X_1, \quad F_{21} = X_1^{-1}X_2.
\]
\end{proof}

The theorem implies  that we cannot resolve the singularities from the singularities of $Xf$ near $\gamma_0$. 

\begin{corollary}\label{cor_cancel} 
Given $f_1\in \mathcal{E}'(M_0)$ with $\WF(f_1)\subset V^1$, there exists $f_2\in \mathcal{E}'(M_0)$ with $\WF(f_2)\subset V^2$ so that $X(f_1+f_2)\in C^\infty(\mathcal{V})$. 
\end{corollary}
In other words, we can cancel any singularities close to $(p_1,\xi^1)$ with a suitable chosen ``annihilator'' with a wave front set near $(p_2,\xi^2)$. 

\begin{remark}
The results so far can be easily extended to $f=f_1+\dots+f_m$, where $f_j$ are microlocally supported near $N^*\gamma$ over neighborhoods of $p_1, \dots, p_m$, conjugate to each other. Then $Xf$ is microlocally equivalent to $X_1f_1+\dots + X_m f_m$, with all $X_j$ elliptic FIOs associated to canonical graphs of diffeomorphisms, if $\kappa\not=0$. Given $f_1, \dots,f_{m}$ with the exception of $f_i$ with $i$ fixed, one can find $f_i$ which annihilates all singularities by simply inverting $X_i$. 
\end{remark}

So far, we assumed that we know $Xf$ locally, and in that in particular \r{odd} holds. If $\kappa(x,\theta)$ is an even function of $\theta$, replacing $\gamma(t)$ by $\gamma(-t)$ provides the same information. 
%Combining this with standard arguments relating resolution of singularities and stability, see, e.g., \cite{SU-JFA}, we get the following theorem. 
We can formulate the following global results. Part (a) of the theorem below is  essentially known, see, e.g., \cite{SU-JFA,SU-AJM}.

\begin{theorem}\label{thm_main1} 
Let $(M_1,g)$ be a  Riemannian manifold   and let $M\Subset M_1$ be compact submanifold with boundary.  Assume that all geodesics having a common point with $M$ exits $M$ at both ends (i.e., $M$ is non-trapping). Let $\mathcal{M}_0$ be an open set of geodesics so that 
\be{5.1c}
\cup_{\gamma\in\mathcal{M}_0}N^*\gamma\supset T^*M.
\ee
Then

(a) If there are no conjugate points on any maximal geodesic in $M_1$, then there exists $s_1$ so that
%the geodesic X-ray transform $X$ in $M$ is injective and stable:
\be{5.1cc}
\|f\|_{H^{s-1/2}(M_1)} \le C_s\|Xf\|_{H^s(\mathcal{M}_0)} + C_s \|f\|_{H^s(M_1)}, \quad \forall s\ge1/2, \quad \forall f\in C^\infty(M).
\ee
In particular, if $X$ is injective on $C_0(M)$, then then we have
\be{5.1c2}
\|f\|_{H^{s-1/2}(M_1)} \le C\|Xf\|_{H^{s}(\mathcal{M})}.
\ee

(b) Let $\kappa(x,\theta) = \kappa(x,-\theta)$. If at least one geodesic in $M_1$ has conjugate points in the interior of $M$, then the following estimate does not hold, regardless of the choice of $s_1$, $s_2$, $s_3$:
\be{5.2}
\|f\|_{H^{s_1}(M_1)} \le C\|Xf\|_{H^{s_2}(\mathcal{M}_0)} + C\|f\|_{H^{s_3}({M_1})}.
\ee
\end{theorem}

\begin{remark}
Here, we define the Sobolev spaces on $\mathcal{M}$ via a fixed coordinate atlas. An invariant definition is also possible, for example by identification of $\mathcal{M}$ and a subset of the unit ball bundle $BM_1$, where $M_1\Supset M$ as above. 
\end{remark}

\begin{remark}
We do not assume convexity here. If a geodesic $\gamma$ has a few disconnected segments in $M$, (b) is still true since we regard the integral $Xf(\gamma)$ as one number. On the other hand, if we know the integral over each segment, we may be able to resolve the singularities, if (a) applies to each one, and then (b) may fail. 
\end{remark}

 \begin{remark}
Condition \r{5.1c} says that for each $(x,\xi)$ there is at least one (directed) geodesic through $x$ normal to $\xi$. This, of course, means that if we ignore the direction, all geodesics through $M$ must be in $\mathcal{M}_0$. The direction however matters because $\kappa$ is not necessarily  even w.r.t. $\theta$. In (c), condition \r{5.1c} simply means that $\mathcal{M}$ must contain all geodesics through $M$. 
 \end{remark}
  \begin{remark}
It is easy to prove that under the assumptions in (b), even the weaker conditional type of estimate of the type \r{5.2}, with $\|Xf\|_{H^{s_2}(\mathcal{M}_0)}$ replaced by $\|Xf\|^\alpha_{H^{s_2}(\mathcal{M}_0)}  \|f\|^\beta_{H^{s_4}(M_1)} $, $\alpha+\beta=1$, does not hold. 
 \end{remark}

\begin{proof}

 Under the assumptions of (a), $X$ is an elliptic FIO of order $-1/2$ associated to the canonical graph $C$. Given $(x,\xi)\in T^*M\setminus 0$, restrict $X$ to a small neighborhood of the geodesic normal to $\xi$ at $x$, which existence is guaranteed by \r{5.1c}. Then $\mathcal{C}$ is a diffeomorphism, by Theorem~\ref{thm_2D}. 
 It therefore has a parametrix, and \r{5.1cc} follows for $Pf$ on the left, where $P$ is a \PDO\ of order $0$ essentially supported neat $(x,\xi)$. Using a partition of unity, we prove \r{5.1cc}.     Estimate \r{5.1c2} follows from \r{5.1c} and a functional analysis argument, see, e.g., \cite[Proposition~5.3.1]{Taylor-book0}.
 
Part (b) follows from Theorem~\ref{thm_cancel} and \cite{SU-JFA}, for example. 
\end{proof}

We note that estimate \r{5.2} can be microlcalized with pseudo-differential cutoffs in an obvious way; and the same conclusion remains. 

\medskip 
\textbf{Properties of $X^*X$}.  The operator $X^*$ is often used as a first step of an inversion of $X$, and it was the main object of study in \cite{SU-caustics}. We show the connection next between the theorems above and the properties of $X^*X$. 

With the microlocalization above, $X=X_1+X_2$. Then
\[
X^*X = X_1^*X_1 + X_1^*X_2+X_2^*X_1+X_2^*X_2.
\]
The operators above are FIOs %($X_1^*X_1$ and $X_2^*X_2$ are \PDO s) 
with canonical relations with the following mapping properties: $X_1^*X_1$ and $X_2^*X_2$ are \PDO s acting in $V^1$ and $V^2$, respectively, $X_1^*X_2$ maps $V^2$ to $V^1$ microlocally, and its adjoint $X_2^*X_1$ maps $V^1$ to $V^2$ microlocally. If $\kappa(p_2,\xi^2)\not=0$, then $X_2^*X_2$ is elliptic if $V^2$ is small enough, and for any \PDO\ $P_2$ essentially supported in $V^2$, and for $f=f_1+f_2$ as above, we have
\[
 (X_2^* X_2)^{-1}  P_2 X^*X f = f_2 + (X_2^* X_2)^{-1}  P_2  X_2^*X_1f_2,
\]
where the inverses are parametrices, and the equality is modulo smoothing operators applied to $f$. Then we get \r{5.1a} again with 
\[
F_{21}= (X_2^* X_2)^{-1}  P_2  X_2^*X_1.
\]
This is natural --- we obtained $F_{21}= X_2^{-1}X_1$ in Theorem~\ref{thm_cancel}; and one way to construct $X_2^{-1}$ is to apply $X_2^*$ first, and then to apply $(X_2^* X_2)^{-1}$. In particular, we get the following generalization of one of the main results in \cite{SU-caustics} to the 2D case.

\begin{theorem}\label{thm_xstar}
With the notation and the assumptions above, 
\[
\begin{split}
X^*Xf &= A_1\left( f_1 +  F_{12}f_2\right) \quad \text{microlocally in $V^1$},\\
X^*Xf &= A_2\left( f_2 +  F_{21}f_1\right) \quad \text{microlocally in $V^2$},
\end{split}
\]
where $A_1=X_1^*X_1$ and $A_2=X_2^*X_2$ are \PDO s with principal symbols 
\be{symb1}
\frac\pi{|\xi|}\left( |\kappa_j(x,\xi_\perp/|\xi|)|^2  +   |\kappa_j(x,-\xi_\perp/|\xi|)|^2        \right),
\ee
 $j=1,2$, respectively, and $F_{21}$, $F_{12}$ are the FIOs of Theorem~\ref{thm_cancel}. 
\end{theorem}

Note that under assumption \r{odd}, one of the two summands above vanishes. 

This theorem explains the artifacts we will get when using $X^*Xf$ (times an elliptic \PDO) as an attempt to recover $f$. Assume that $f_2=0$ and we want to recover $f_1$ from $Xf_1$, and assume that $\kappa\not=0$. We will get the original $f_1$ (in $V^1$) plus the ``artifact'' $F_{21}f_1$ (in $V^2$). This is not a downside of that particular method;  by Theorem~\ref{thm_cancel}, we cannot say which is the original. The true image might be either $f_1$ or $F_{21}f_1$, or some convex linear combination of the two. All we can recover is $f_2+ F_{21}f_1$, or, equivalently, $f_1+F_{12}f_2$.

\section{Recovery of singularities and stability for certain non-even weights. The attenuated transform}\label{sec5a}
\subsection{Heuristic arguments}
The analysis so far was concerned with whether we can recover $\WF(f)$ near $N^*\gamma_0$ 
%(at $\{N^*\gamma;\;  \gamma\in\mathcal{M}_0\}$) 
from the knowledge of the singularities of $Xf$ known near $\gamma_0$. We allowed for the geodesic set $\mathcal{M}_0$ to be not necessarily small but we assumed \r{odd}. On the other hand, each $(x,\xi)\in T^*M\setminus 0$ can be possibly detected by the two directed geodesics (and the ones close to them) through $x$ normal to $\theta$. If $Xf$ is known there, we have two equations for each such $(x,\xi)$. Let us say that we are in the situation in the previous section. Instead of the single equation \r{5.1'}, we have two equations, each one coming from each direction. We think of this as a $2\times 2$ system, and if its determinant is not zero, loosely speaking, we could solve it and still recover the singularities. 

In the situation in the previous section, let $\mathcal{C}(p_1,\xi^1) = \mathcal{C}(p_2,\xi^2)$. Let 
 $(p_1,v_1)$ and $(p_2,v_2)$, respectively be the corresponding unit tangent vectors, pointing in the same direction along the geodesic connecting $p_1$ and $p_2$. 
%and the latter are tangent to the same geodesic, close to $\gamma_0$, passing through $p_1$ and $p_2$. %Denote one of those geodesics by $\gamma_+$, and the one with the opposite orientation by $\gamma_-$. 
Without loss of generality, we may assume that $\xi^1/|\xi^1|=v_1^\perp$ (the other option is to be $-v_1^\perp$). Then $ \eps \xi^2/|\xi^2|=v_2^\perp$, where $\eps=(-1)^m$, as above, where $m$ is the number of conjugate points between $p_1$ and $p_2$.

Let $X_+$ and $X_-$ be $X$ restricted to a neighborhood of $\gamma_0$ with a positive and a negative orientation, respectively. 
%of $\gamma_-$ and $\gamma_+$, respectively. 
Then 
\[
X_+(f_1+f_2)\in C^\infty, \quad X_-(f_1+f_2)\in C^\infty
\]
is a $2\times 2$ system. Heuristically, $(p_1,v_1)$ is involved there with weight $\kappa(p_1,v_1)$ in the first equation, and with weight $\kappa (p_1,-v_1)$ in the second one; similarly for $(p_2,v_2)$. Therefore, if 
\be{R1}
\det
\begin{pmatrix}
 \kappa(p_1,v_1)   &    \kappa(p_2,v_2)          \\
 \kappa(p_1,-v_1)   & \kappa(p_2,-v_2) 
  \end{pmatrix}\not=0,
\ee
the system should be solvable and we should be able to recover the singularities at those points.

One such case is the attenuated geodesic X-ray transform with a positive attenuation. The weight $\kappa$ then decreases strictly along the geodesic flow. Then $\kappa(p_1,v_1)> \kappa(p_2,v_2)$, and  $\kappa(p_2,-v_2)< \kappa(p_2,-v_2)$, and then the determinant is positive, therefore not zero. 

Finally, if we want to resolve three singularities placed at three conjugate to each other points $p_1$, $p_2$, $p_2$, we would have two equations for three unknowns, and recovery would not be possible. 

\subsection{Recovery of singularities for non-even weights} 

\begin{theorem}\label{thm5.1}
Let $f=f_1+f_2$ with $\WF(f_k)\subset V^k$, $k=1,2$. 
 If $\kappa\not=0$,  \r{R1} holds, 
  and $V^1$ and $V^2$ are small enough, then $X(f_1+f_2)\in H^s(\mathcal{V})$ implies $f_k\in H^{s-1/2}(V^k)$, $k=1,2$.

%(b) If $\kappa(p,v)\kappa(q,-w) = \kappa(q,w)\kappa(p,-v)$ for $(p,v,q,w)$ in some neighborhood of $(p_1,v_1,p_2,v_2)$, then the conclusion of Corollary~\ref{cor_cancel} holds instead. 
\end{theorem}

\begin{proof} We assume first that $\WF(f_1)\subset V^1_+$; then $\WF(f_2)\subset V^2_{-\eps}$. The case $\WF(f_1)\subset V^1_-$ is similar. All inverses below are parametrices and all equalities hold modulo $C^\infty$ in either $\mathcal{V}$ or $V^1_-$ or $ V^2_{-\eps}$, depending on the context.  
With the notation of the previous section, let $X_{+,k}$ be $X_+$ restricted to a neighborhood  of $p_k$, $k=1,2$, and we similarly define $X_{-,k}$. Then \r{5.1'} gives us two equations
\be{R2}
X_{+,1}f_1+ X_{+,2}f_2= g_+:=X_+f \in H^s(\mathcal{V}), \quad X_{-,1}f_1+ X_{-,2}f_2= g_-:= X_-f\in H^s(\mathcal{V})
 \ee 
Since both $\kappa(p_2,v_2)$ and $\kappa(p_2,-v_2)$ are non-zero by assumption, $X_{\pm,2}$ are elliptic, and we have
\[
P f_1 = X_{+,2}^{-1}g_+ - X_{-,2}^{-1}g_-, \quad P:= X_{+,2}^{-1}X_{+,1}-X_{-,2}^{-1} X_{-,1}.
\]
We have
\be{X}
P =    X_{+,2}^{-1}X_{+,1}( \Id -Q), \quad Q:= X_{+,1}^{-1} X_{+,2} X_{-,2}^{-1} X_{-,1}  .
\ee
The operator $ X_{+,2} X_{-,2}^{-1}$ is a \PDO\ on $\mathcal{M}_0$. To compute its principal symbol, write 
\be{XX}
 X_{+,2} X_{-,2}^{-1} =X_{+,2}\left( X_{-,2}^{-1} X_{+,2}\right)X_{+,2}^{-1}.
 \ee
  %Assume that the orientation along $\gamma$ is chosen so that the first term in \r{symb1} is the non-zero one near $p_2$. Then
  The term $ X_{-,2}^{-1} X_{+,2}$ in the parentheses in \r{XX}  is a \PDO\ on $M$. 
  Then the principal symbol of $ X_{-,2}^{-1} X_{+,2}$ is $\sigma_2:= \kappa(x,\eps \xi_\perp/|\xi|)/\kappa(x,-\eps \xi_\perp/|\xi|)  $, as it follows from 
 Theorem~\ref{thm_xstar},  its obvious generalization to operators with different weights, and the construction of $X_{-,2}^{-1}$ by applying $X_{-,2}^*$ to $X_{-,2}$  first. Recall that $ X_{-,2}^{-1} X_{+,2}$ has essential support near $(p_2,\pm\xi^2)$. Then by \r{XX} and by Egorov's theorem, the principal symbol of  $X_{+,2} X_{-,2}^{-1}$ is $\sigma_2\circ \mathcal{C}_2^{-1}$.   In a similar way, by  Theorem~\ref{thm_xstar},  the principal symbol of $X_{-,1} X_{+,1}^{-1} $ is $\sigma_1 =  \kappa(x,-\xi_\perp/|\xi|)/\allowbreak\kappa(x,\xi_\perp/|\xi|) $.
 
 We can now write, see \r{X},\
\[
Q =  X_{-,1}^{-1} \left(X_{-,1} X_{+,1}^{-1} \right)  \left(X_{+,2} X_{-,2}^{-1}\right) X_{-,1}  .
\]
The operator in the first parentheses, $X_{-,1} X_{+,1}^{-1} $, has principal symbol $\sigma_1\circ \mathcal{C}_1^{-1}$.  Applying Egorov's theorem again, we get that the principal symbol of $Q$ is
\[
\left[\left(\sigma_1\circ \mathcal{C}_1^{-1}\right)  \left( \sigma_2\circ \mathcal{C}_2^{-1}\right)\right]\circ \mathcal{C}_1 = \sigma_1 \left( \sigma_2\circ \mathcal{C}_{21}\right). 
\]
Since $\mathcal{C}_{21}(p_1,\xi^1)= (p_2,\xi^2)$, at $(p_1,\xi^1)$, $Q$ has principal symbol 
\[
\sigma_p(Q)(p_1,\xi^1) =  \frac{\kappa(p_1,-\xi^1_\perp/|\xi^1|))    }{\kappa(p_1,\xi^1_\perp/|\xi^1|} 
\frac{\kappa(p_2,\eps \xi^2_\perp/|\xi^2|))    }{\kappa(p_2,-\eps\xi^2_\perp/|\xi^2|} 
=   \frac{\kappa(p_1,-v_1) \kappa(p_2,v_2) }{ \kappa(p_1,v_1) \kappa(p_2,-v_2)}.
\]
Then $\Id-Q$ in \r{5.2} is elliptic, and $P$ is a composition of an elliptic FIO of oder zero associated with $\mathcal{C}_{21}$ and an elliptic \PDO\ (recall that all operators are microlocalized in small enough conic sets)   of order $0$ if $\sigma_p(Q)(p_1,\xi^1)\not=1$. %This proves (a). 
%Part (b) follows from Corollary~\ref{cor_cancel} since $X_-$ provides no 
\end{proof}

\begin{remark}
We actually proved that $f=f_1+f_2$ can be recovered microlocally in $V^1\cup V^2$   by
\be{par}
\begin{split}
f_1 &=\left( X_{+,2}^{-1}X_{+,1}-X_{-,2}^{-1} X_{-,1}\right)^{-1} \left(X_{+,2}^{-1}g_+ - X_{-,2}^{-1}g_-\right),\\
f_2 &=\left( X_{+,1}^{-1}X_{+,2}-X_{-,1}^{-1} X_{-,2}\right)^{-1} \left(X_{+,1}^{-1}g_+ - X_{-,1}^{-1}g_-\right),
\end{split}
\ee
given $g_\pm = X_\pm f$. As above, all inverses are microlocal parametrices. In terms of $Q$, this could be written as
\be{par2}
f_1 =\left( \Id-Q\right)^{-1}  X_{+,1}^{-1} \left(g_+ -    X_{+,2}  X_{-,2}^{-1}g_-\right),
\ee
with a similar formula for $f_2$ (which can be also reconstructed by solving the first or the second equation on \r{R2} for $f_2$). 
\end{remark}

\subsection{The attenuated X-ray transform} An important example satisfying \r{R1} is the attenuated X-ray transform. It as a weighted transform with weight
\be{P1}
\kappa (x,v) = e^{-\int_0^{\infty}\sigma(\gamma_{x,v}(s), \dot \gamma_{x,v}(s)) \, \d s},
\ee
where $\sigma(x,v)\ge0$ is the attenuation. 
The weight increases along the geodesic flow. More precisely, if $G$ is the generator of the geodesic flow, we have
\[
G\log \kappa = \sigma\ge0. 
\]
If $\sigma>0$ along $\gamma_{x,v}$, then $\kappa$ is strictly increasing. Then \r{R1} is trivially satisfied. In fact, for a fixed $(p_1,v_1,p_2,v_2)$, \r{R1} is equivalent to requiring that $\sigma>0$ for at least one point on the geodesic through these points, between $p_1$ and $p_2$. 
 We then get the following.

\begin{theorem}
Let  $(M_1,g)$ and $M$ be as in Theorem~\ref{thm_main1}. 
  Assume that    $C^\infty\ni\sigma>0$ in $M$ and let $X$ be  the attenuated X-ray transform related to $\sigma$. 

(a) If there are at most two conjugate points along each geodesic through $K$, then $X$ is stable, i.e.,  the conclusions of Theorem~\ref{thm_main1}(a) hold.

(b) If there is a geodesic through $K$ with three (or more) conjugate points, then there is no stability, i.e., the conclusion of Theorem~\ref{thm_main1}(b)  holds. 
\end{theorem}

\begin{proof}
Part (a) follows directly from \r{par} and a partition of unity on $S^*M$, as in the proof of Theorem~\ref{thm_main1}.

%\subsection{Three or more conjugate points} 
If there are three or more points on $\gamma_0$ conjugate to each other, recovery of singularities is impossible. Indeed, let those points and the corresponding unit tangent vectors are $(p_k,v_k)$, $k=1,\dots,N$, $N\ge3$. We define the canonical relations $\mathcal{C}_{kl}= \mathcal{C}_k^{-1}\circ \mathcal{C}_l$ as before, and let $V^k$ be conic neighborhoods of the conormals at $p_k$. Then $\mathcal{C}_k(V^k)=\mathcal{V}$ for all $k$. Let $f_k$ have wave front sets in $V^k$. Given any $N-2$ of them, for example $f_3,\dots f_N$, we can solve
\[
\begin{split}
X_{+,1}f_1+ X_{+,2}f_2+ X_{+,3}f_3+\dots&=g_+,\\
X_{-,1}f_1+ X_{-,2}f_2+ X_{-,3}f_3+\dots&=g_-
\end{split}
\]
for $f_1$ and $f_2$ microlocally. Therefore, we cannot recover the singularities and if, in particular,   $X(f_1+\dots+f_N)\in C^\infty(\mathcal{V})$, the last $N-2$ distributions can have arbitrary singularities on $V^k$, $k=3,\dots,N$. This implies (b), see, e.g., \cite{SU-JFA}. 

\end{proof}

%\newpage
\section{Numerical Experiments}\label{sec5}
\subsection{Cancellation of singularities} 
%\subsection{Construction of the conjugate image of a peaked Gaussian}
The first numerical experiment aims to illustrate Theorem~\ref{thm_cancel}. We choose $f_1$ and construct $f_2$ so that $f_1-f_2$ has fewer singularities than $f_1$. 

We use the code developed in \texttt{Matlab} by the first author, whose detail may be found in \cite{Monard2013}. The manifold $M$ is chosen to be the unit disk while the smaller neighborhood $U_2$ where the ``artifacts'' are expected, 
  is the disk of center $(0,0.5)$ and radius $0.5$ (both domains are displayed at the left of Fig. \ref{pic1}). We pick the (isotropic) metric from \cite{Monard2013}, taking the scalar expression
\begin{align}
    g(x,y) = \exp \left( k\exp \left( -\frac{x^2+y^2}{2\sigma^2} \right) \right), \quad \text{with } \quad \sigma = 0.25, \quad k= 1.2.
    \label{eq:lensmetric}
\end{align}
The manifold $(M,g)$ is not simple while the manifold $(M_2,g)$ is. We choose $f_1$ to be a Gaussian well concentrated near a single point, and we view this as an approximation of a delta function. The thick-marks in Figure~\ref{pic3}, left, show the mesh chosen on the circle. The discretization of the initial directions is not visualized. The X-ray transform $Xf_1$ is supported on the ingoing boundary of $M$ and is parameterized in so-called ``fan-beam'' coordinates $(\beta,\alpha)\in [0,2\pi]\times[-\frac{\pi}{2},\frac{\pi}{2}]$, where $\beta$ locates a the initial (boundary) point $\gamma(0) = \binom{\cos\beta}{\sin\beta}$ of the geodesic and $\alpha$ denotes the argument of its (unit) speed with respect to the inner normal, i.e. $\dot\gamma(0) = g^{-\frac{1}{2}} (\gamma(0)) \binom{\cos(\beta+\pi+\alpha)}{\sin(\beta+\pi+\alpha)}$. %Note that this rectangle is in fact a representation of a torus. 
The X-ray transform of a delta function is a delta function on a certain curve in $\mathcal{M}$. Its conjugate locus is above its center, with two folds connected in  a cusp, see the second plot in  Figure~\ref{pic4}. The artifacts in the reconstruction of the delta should be supported above the conjugate locus and conormal to the fold part of it. 

\begin{figure}[!ht] % float placement: (h)ere, page (t)op, page (b)ottom, other (p)age
  \centering
  \includegraphics[trim = 60 40 60 20, clip, width = \textwidth]{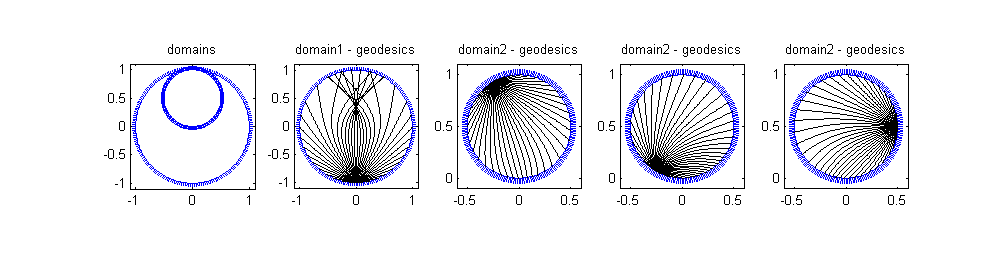}
  \caption{The computational domain and a few family of geodesics. The large disk is $M$ labeled as \texttt{domain1}, the smaller one is $U_2$, called \texttt{domain2}. $(M_1,g)$ has circular symmetry, so the geodesics from any other boundary point look similar on the second picture. $(U_2,g)$ has no symmetry, though it is simple (notice axes have changed on the three rightmost pictures, to fit the size of {\tt domain2}).} 
%On the second row, a few families of geodesics in \texttt{dom2} are plotted and the domain is enlarged to the size of \texttt{domain1}.  }
  \label{pic1}
\end{figure}
  
\begin{figure}[!ht] % float placement: (h)ere, page (t)op, page (b)ottom, other (p)age
  \centering
  \includegraphics[height=0.22\textheight]{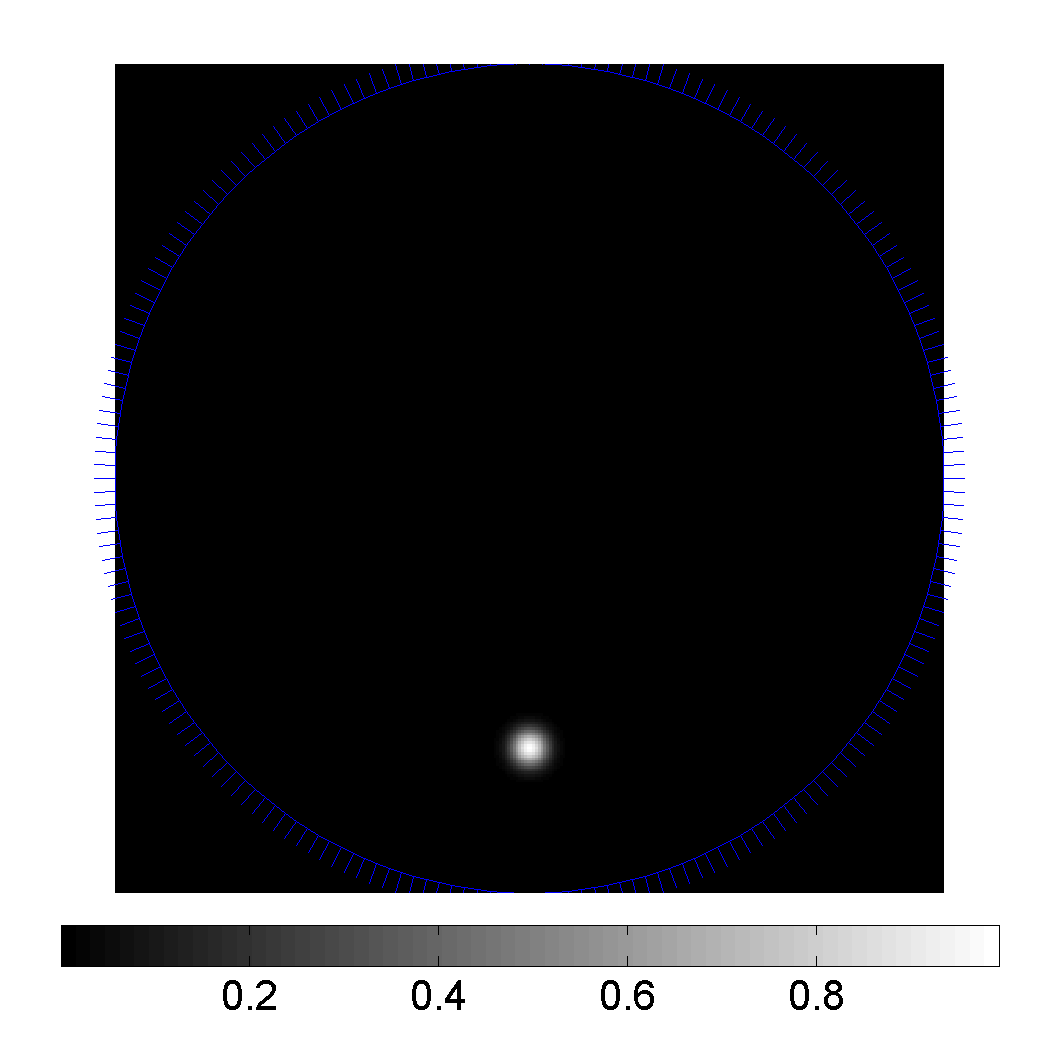}\ \ \ \ 
  \includegraphics[height=0.22\textheight]{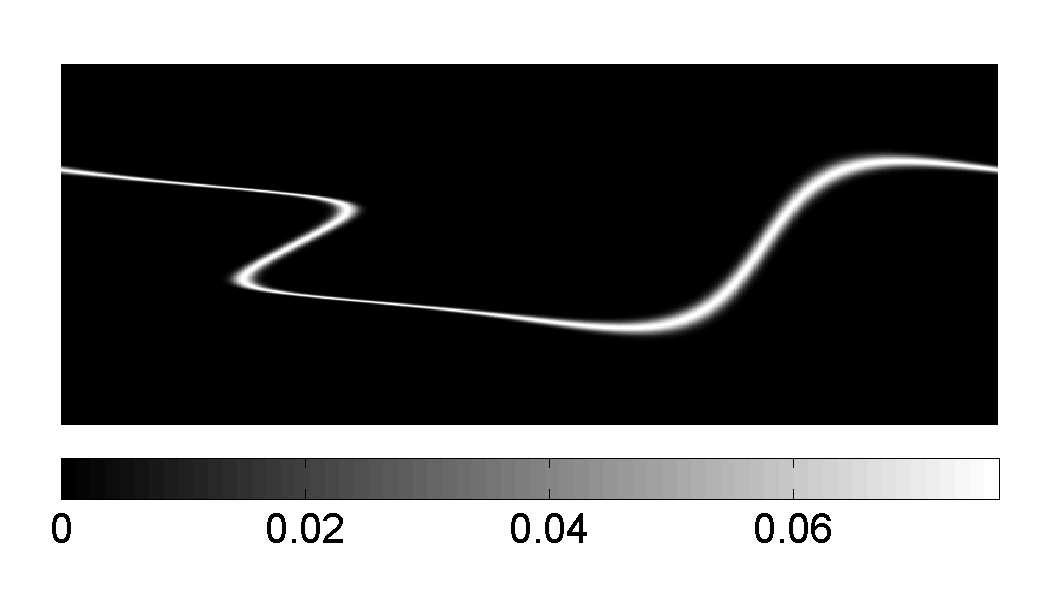}
\caption{The function $f_1$ (left) and $Xf_1$ (right). Observe on $Xf_1$ that when the metric is non-simple, the X-ray transform of a delta function can no longer be expressed as a one-to-one $\alpha(\beta)$ function.}  
  \label{pic3}
\end{figure}

The goal next is to compute $F_{21}f_1=X_2^{-1}X_1f_1$ microlocally, i.e., to construct a function  with a wave front set as that of $F_{21}f_1$ near the conjugate locus. 
We compute $Xf_1$ first in $M$ and then we remap the data from $M$ to $U_2$ via free geodesic transport. 
The so remapped data does not fill the whole $\partial_+SU_2$ and 
 may not belong to the range of $X$ on $U_2$ there. Still, a microlocal inversion is possible. On the remapped data, we apply the reconstruction formula on $U_2$ derived in \cite{PestovU04} and implemented in \cite{Monard2013}, and call the reconstructed quantity $ f_2$. This is equivalent, microlocally, to computing $X_1f$ first and and applying $X_2^{-1}$ to that, i.e., the result is $ f_2= X_2^{-1}X_1f_1$ on some conic open set $V^2$ as above. 
 Numerically, computing $X_2^{-1}X_1f_1$ is based on the following. We can reduce the problem $X_2f_2= X_1f_1$ (microlocally) 
 %of inverting $X_2$ in $U_2$ 
 to solving  the Fredholm-type equation of the form
\[ f_2 + W^2 f_2 = AX_1 f_1, \]
where $W^2$ is an operator of order $-1$ when the metric is simple, whose Schwartz kernel is expressed in terms of the Jacobi fields $(a,b)$ as  section~\ref{sec_3.2}, but with Cauchy conditions on the boundary $(1,0)$ and $(0,1)$, respectively, and $A$ is an explicit approximate reconstruction formula. It is proved in \cite{Krishnan2010} that $W^2$ is a contraction when the metric has curvature close enough to constant, though numerics in \cite{Monard2013} indicate that considering $W^2$ a contraction and inverting the above equation via a Neumann series successfully reconstructs a function from its ray transform in all simple metrics considered. Once $f_2$ is constructed using this approach, we subtract it from $f_1$ (Figure~\ref{pic4}, left), then compute the forward data $X(f_1-f_2)$ on the large domain (see Fig.~\ref{pic5}, left, where some singularities of $Xf_1$ have been canceled). The function/distribution $f=f_1-f_2$, plotted in Figure~\ref{pic4}, is then  the one with  canceled singularities, by Theorem~\ref{thm_cancel}.  
Figure~\ref{pic5} illustrates the cancellations. Of course, only some  open conic set of the singularities is canceled, corresponding to geodesics having conjugate points in $M$. In fact, it is clear from Figure~\ref{pic4} that the cancellation occurs near two directed vertical geodesics corresponding to s small strip around the horizontal medium in Figure~\ref{pic5}.

%In Figure~\ref{pic3}, we plot $f=f_1-f_2$, which is the function/distribution with the canceled singularities.  

\begin{figure}[!ht] % float placement: (h)ere, page (t)op, page (b)ottom, other (p)age
  \centering
  \includegraphics[height=0.24\textheight]{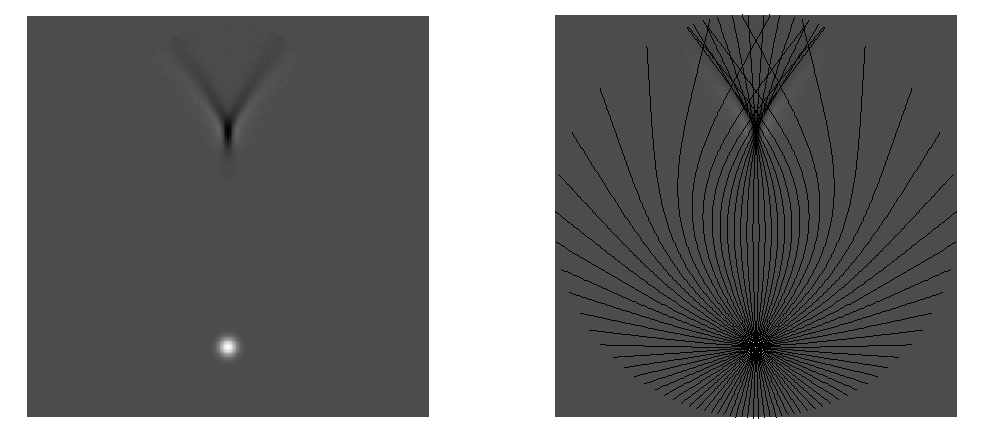}
\caption{The function $f=f_1-f_2$ (left) and the same function with a few super\-imposed geodesics on it (right). The ``artifact'' $f_2$ appears as an approximate conormal distribution to the conjugate locus of the blob that $f_1$ represents. The gray scale has changed, and black now represents negative values, around $-0.5$.}  
  \label{pic4}
\end{figure}
 
\begin{figure}[!ht] % float placement: (h)ere, page (t)op, page (b)ottom, other (p)age
  \centering
  \includegraphics[width=0.48\textwidth]{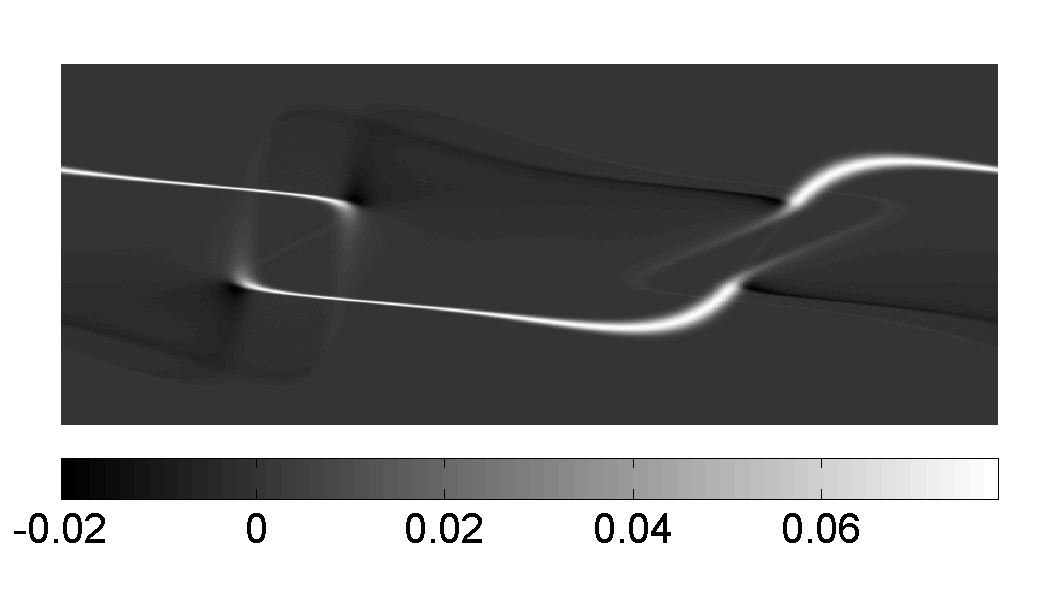}\ \ \ \ 
  \includegraphics[width=0.48\textwidth]{03_If}  
\caption{$X(f_1-f_2)$ (left) and $Xf_1$ (right). Some singularities of $Xf_1$ are nearly erased. The gray scale on the left is slightly different to allow for the negative values of  $X(f_1-f_2)$.}  
  \label{pic5}
\end{figure}

%\newpage
\subsection{Artifacts in the reconstruction}  %{Direct computation of the normal operator}
We illustrate Theorem~\ref{thm_xstar} now; what happens if we use $X^*X$ as an reconstruction attempt. If we apply this to the previous setup, we would get $f_1+f_2$ rather than $f_1-f_2$. 
 Here we still consider the domain $M$ with the metric from \eqref{eq:lensmetric} translated so that it is centered at $(0.2,0)$. Now, $f_1$ is a collection of peaked Gaustsians alternating in signs (Fig.~\ref{pic6}, left). Set $N=X^*X$. We apply $X^*$ to $Xf$, and then $N$ again to get $N^2f$. The advantage to this is that locally near $p_1$ and near $p_2$, the parametrix is $-C\Delta_g$ instead of a square root of the latter. Then we apply $-C\Delta_g$ to get $-C\Delta_g N^2f$. Near $f_1$, this recovers $f_1$ up to an operator of order $-1$ applied to it. It also ``recovers'' the artifact $f_2$. The results are shown in Figure~\ref{pic6}. 

\begin{figure}[h] % float placement: (h)ere, page (t)op, page (b)ottom, other (p)age
  \centering
  \includegraphics[height=0.31\textheight]{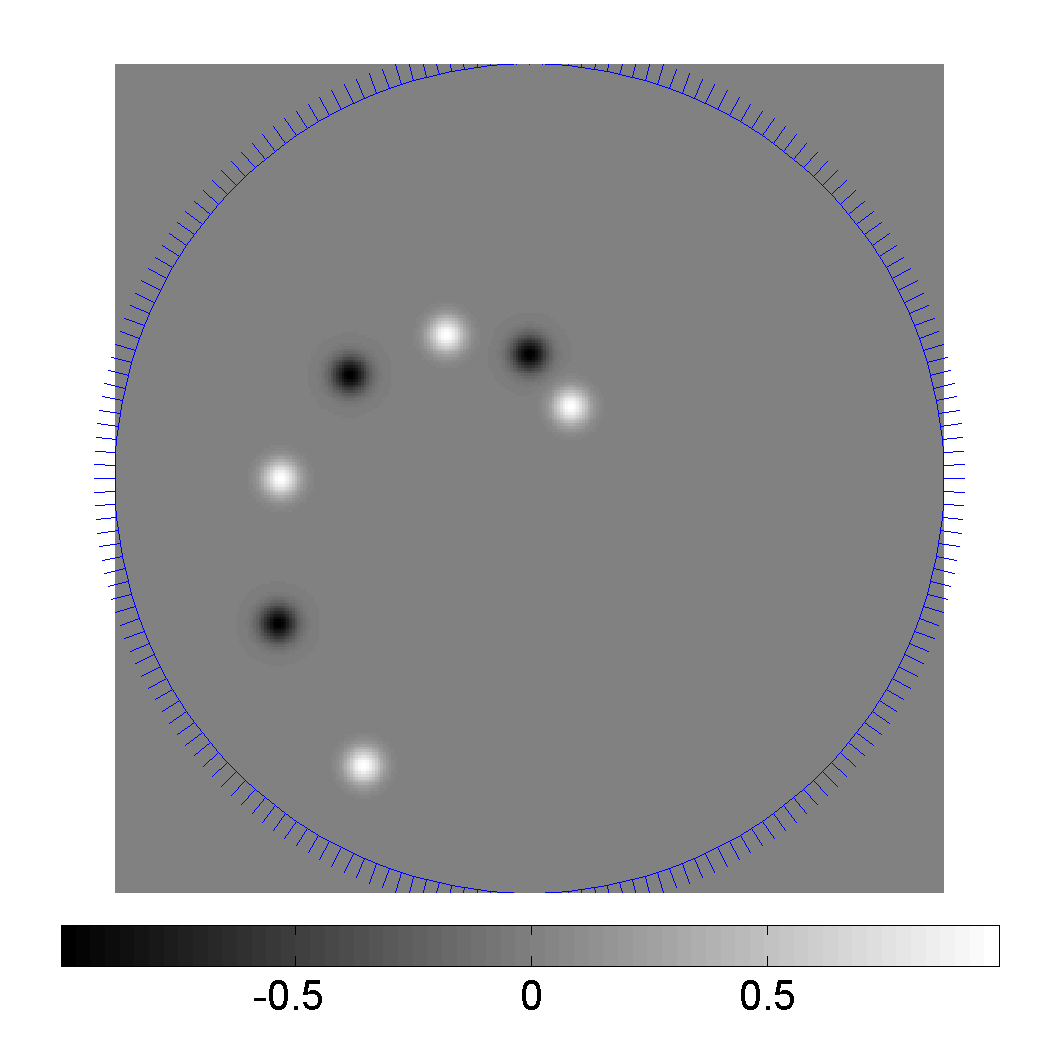}\ \ \ \ 
  \includegraphics[height=0.31\textheight]{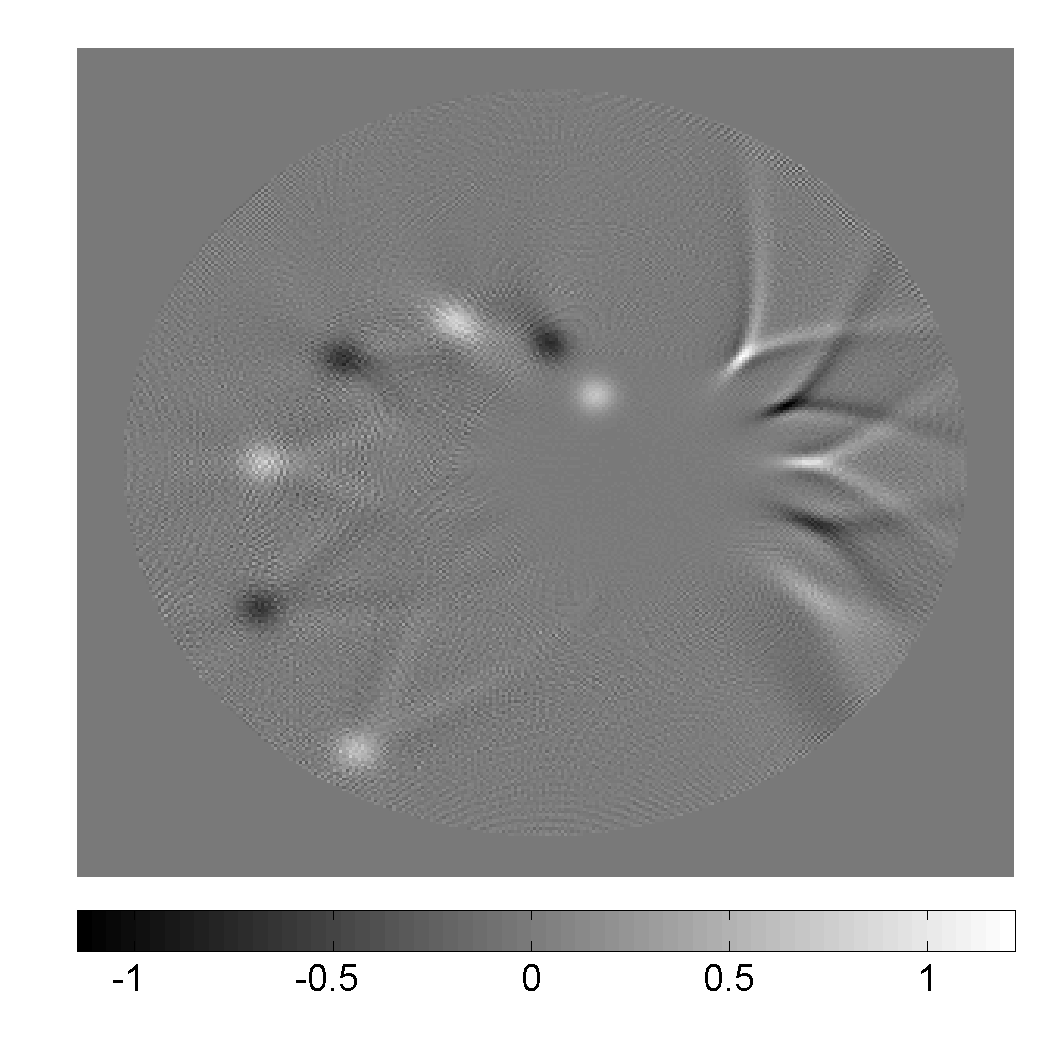}
\caption{$f_1$ (left) and $-C\Delta_g N^2 f_1$ (right).}  
  \label{pic6}
\end{figure}

The artifacts appear as an approximation of the union of the conjugate loci of each blob. Unlike the previous example, we see here $f_1+f_2$ (what we recover), not $f_1-f_2$ (what would cancel the singularities). The two blobs closer to the center create no artifacts because their conjugate loci are out of the disk $M$.

%\newpage
%\bibliographystyle{abbrv}
%\bibliography{myreferences}
%\end{document}

%
\end{document}